\documentclass{amsart}

\usepackage[all]{xy}
\usepackage{hyperref}
\hypersetup{
    colorlinks,%
    citecolor=blue,%
    filecolor=blue,%
    linkcolor=blue,%
    urlcolor=blue
}
\usepackage{booktabs}
\usepackage{mathrsfs}
\usepackage{todonotes}
\usepackage[backrefs,lite]{amsrefs}
\usepackage{longtable}

\usepackage{enumitem}
\usepackage{graphicx}

\usepackage{tikz}
\usepackage{wrapfig}

\newtheorem{introthm}{Theorem}

\newtheorem{theorem}{Theorem}[section]
\newtheorem{lemma}[theorem]{Lemma}
\newtheorem{proposition}[theorem]{Proposition}

\theoremstyle{definition}

\newtheorem{remark}[theorem]{Remark}
\theoremstyle{remark}

\numberwithin{equation}{section}

\def\cc{{\mathbb C}}

\def\zz{{\mathbb Z}}

\def\qq{{\mathbb Q}}
\def\pp{{\mathbb P}}

\def\Osh{{\mathcal O}}
\def\R{{\mathcal R}}

\def\Eff{\operatorname{Eff}}
\def\Nef{\operatorname{Nef}}

\def\Mov{\operatorname{Mov}}
\def\Cl{\operatorname{Cl}}

\def\Pic{\operatorname{Pic}}

\def\rk{\operatorname{rk}}

\def\Spec{\operatorname{Spec}}

\def\MW{\operatorname{MW}}

\begin{document}

\title{On del Pezzo elliptic varieties
of degree $\leq 4$}

\author[A.~Laface]{Antonio Laface}
\address{
Departamento de Matem\'atica,
Universidad de Concepci\'on,
Casilla 160-C,
Concepci\'on, Chile}
\email{alaface@udec.cl}

\author[A.~Tironi]{Andrea L. Tironi}
\address{
Departamento de Matem\'atica,
Universidad de Concepci\'on,
Casilla 160-C,
Concepci\'on, Chile}
\email{atironi@udec.cl}

\author[L.~Ugaglia]{Luca Ugaglia}
\address{
Dipartimento di Matematica e Informatica,
Universit\`a degli studi di Palermo,
Via Archirafi 34,
90123 Palermo, Italy}
\email{luca.ugaglia@unipa.it}


\thanks{
The first author was partially supported 
by Proyecto FONDECYT Regular N. 1150732.
The second author was partially supported 
by Proyecto VRID N. 214.013.039-1.OIN. The third author was partially supported 
by Universit\`a di Palermo (2012-ATE-0446).
}

\begin{abstract}
Let $Y$ be a del Pezzo variety of degree $d\leq 4$ 
and dimension $n\geq 3$, let $H$ be an ample
class such that $-K_Y=(n-1)H$ and let $Z\subset Y$
be a $0$-dimensional subscheme of length $d$
such that the subsystem of elements of $|H|$ with
base locus $Z$ gives a rational morphism 
$\pi_Z\colon Y\dashrightarrow\pp^{n-1}$. Denote by 
$\pi\colon X\to \pp^{n-1}$ the elliptic fibration obtained 
by resolving the indeterminacy locus of $\pi_Z$.
Extending the results of~\cite{hltu} we study the geometry
of the variety $X$ and we prove that the Mordell-Weil 
group of $\pi$ is finite if and only if the Cox ring of $X$ is 
finitely generated.
\end{abstract}
\maketitle

\section*{Introduction}

Let $Y$ be a del Pezzo variety of dimension $n\geq 3$ 
and $H$ an ample class such that 
$-K_Y=(n-1)H$ and let $d:=H^n$ be the degree of $Y$.
We consider the rational map 
$\pi_Z\colon Y\dashrightarrow\pp^{n-1}$ 
associated to a linear series $V\subset |H|$ of dimension 
$n-1$, having $0$-dimensional base locus $Z$. 
In what follows we say that the map 
$\pi\colon X\to\pp^{n-1}$, obtained by 
resolving the indeterminacy of $\pi_Z$,
is a {\em del Pezzo elliptic fibration} while 
$X$ is a {\em del Pezzo elliptic variety}
of degree $d$.

In~\cite{cps}
the case of general $V$ is considered in relation with
the Morrison-Kawamata cone conjecture. In~\cite{hltu}
the case $\deg(Y)=3$ has been studied, providing the
Mordell-Weil groups of all the types of fibrations that
can be obtained and proving that the group is finite
if and only if the Cox ring of $X$ is finitely generated.

In this paper we extend the results of~\cite{hltu}
to del Pezzo elliptic varieties of degree $\leq 4$.
Our first result is about the Mordell-Weil groups 
of the corresponding del Pezzo elliptic fibrations
(the notation will be explained in Section~\ref{sec:types}).
\begin{introthm}\label{thm:mw}
The Mordell-Weil groups of the
del Pezzo elliptic fibrations of 
degree $d\leq 4$ and dimension $n\geq 3$ 
are the following:
{\small 
\begin{longtable}{|c|l|l|c|l|l|}
\hline
 Degree & Type & $\MW(\pi)$
 &
 Degree & Type & $\MW(\pi)$
 \\[2pt]
\hline
$1$ & $X_{1}$ & $\langle 0\rangle$ & $3$
& $X_{3}, X_S$ & $\langle 0\rangle$\\[2pt]
\hline
$2$ & $X_{11}$ & $\zz$ 
& $4$ & $X_{40}$ & $\zz^3$\\[2pt]
& $X_{SS}$ & $\zz/2\zz$ &
& $X_{41}, X_{30}$ & $\zz^2$ \\[2pt]
& $X_{2}$ & $\langle 0\rangle$ &
& $X_{42}$ & $\zz\oplus(\zz/2\zz)$ \\[2pt]
\cline{1-3}
$3$ & $X_{111}$ & $\zz^2$ &
& $X_{31}, X_{20}, X_{21}$ & $\zz$ \\[2pt]
& $X_{S11}, X_{12}$ & $\zz$ &
& $X_{43}$ & $(\zz/2\zz)^2$\\[2pt]
& $X_{SSS}$ & $\zz/3\zz$ &
& $X_{21}, X_{22}$ & $\zz/2\zz$\\[2pt]
& $X_{S2}$ & $\zz/2\zz$ &
& $X_{10}, X_{11}$ & $\langle 0\rangle$\\[2pt]
\hline
\caption{Mordell-Weil groups of del Pezzo elliptic fibrations}
\label{types}
\end{longtable}
}
\end{introthm}
Our second result is about Cox rings of elliptic del Pezzo
varieties.
\begin{introthm}
\label{cox-mw}
Let $X$ be a del Pezzo elliptic
variety of degree $\leq 4$
and dimension $n\geq 3$. Then the following are equivalent:
\begin{enumerate}
\item the Cox ring of $X$ is finitely generated;
\item the Mordell-Weil group of $\pi\colon X\to \pp^{n-1}$
is finite.
\end{enumerate}
\end{introthm}

We prove Theorem~\ref{cox-mw} showing that any
del Pezzo elliptic variety, whose corresponding elliptic
fibration has finite Mordell-Weil group, is a Mori Dream 
Space and viceversa. Then we conclude by means of 
~\cite{HuKe}*{Proposition 2.9}. The proof of our second 
theorem makes use of a detailed study of the structure of the
moving and effective cones of elliptic del Pezzo varieties.
In particular we prove the following.
\begin{introthm}\label{eff-mov-cones}
Let $\pi\colon X\to\pp^{n-1}$ be a del Pezzo
elliptic fibration of degree $d\leq 4$ and having
finite Mordell-Weil group.
Then the effective cone $\Eff(X)$ is generated by 
the vertical classes and the classes of sections,
the cone $\Mov(X)$ is the dual of $\Eff(X)$ with respect 
to the bilinear form introduced in~\eqref{bil}.
The intersection graphs for the effective cones
are given in the following table,
where each vertex corresponds to a section or a
vertical class $D$, the label in the vertex is $-\langle D,D\rangle$
and the number of edges connecting two vertices
$D$ and $D'$ is 
$\langle D,D'\rangle$. 
\begin{center}
\begin{table}[h]
\begin{tikzpicture}[main node/.style={circle,fill=blue!20,draw,minimum size=3pt,inner sep=1pt}, scale=0.5]

 \begin{scope}[xshift=-8cm,yshift=9.5cm]    
 \foreach \s in {0,-3,19}
 {\draw[-] (\s,0) -- (\s,-18);}
 \foreach \s in {0,-3,-8,-13,-18}
 {\draw[-] (-3,\s) -- (19,\s);}
 \foreach \s/\d in {-1.5/1,-5.5/2,-10.5/3,-15.5/4}
 \node at (-1.5,\s) {$\deg = \d$};
 \end{scope}
 
 \begin{scope}[xshift=-4.8cm,yshift=5cm]    
 \node[main node] (3) at (0,2) {{\tiny\bf 1}};
 \node[main node] (4) at (0,3) {{\tiny\bf 0}};
 \path[every node/.style={font=\sffamily\small}]
    (3) edge[bend right=5]  node {} (4);
\node at (0,4) {\footnotesize $X_1$};
\end{scope}

 \begin{scope}[xshift=-4.8cm,yshift=2cm]    
 \node[main node] (1) at (0,0) {{\tiny\bf 1}};
 \node[main node] (2) at (0,1) {{\tiny\bf 8}};
 \node[main node] (3) at (0,2) {{\tiny\bf 8}};
 \node[main node] (4) at (0,3) {{\tiny\bf 1}};
 \path[every node/.style={font=\sffamily\small}]
    (1) edge[bend right=8]  node {} (2)
    (1) edge[bend right=24] node {} (2)
    (1) edge[bend left=24] node {} (2)
    (1) edge[bend left=8] node {} (2)
    (2) edge[bend right=8] node {} (3)
    (2) edge[bend right=24] node {} (3)
    (2) edge[bend right=40] node {} (3)
    (2) edge[bend right=56] node {} (3)
    (2) edge[bend left=8] node {} (3)
    (2) edge[bend left=24] node {} (3)
    (2) edge[bend left=40] node {} (3)
    (2) edge[bend left=56] node {} (3)
    (3) edge[bend right=8]  node {} (4)
    (3) edge[bend right=24] node {} (4)
    (3) edge[bend left=24] node {} (4)
    (3) edge[bend left=8] node {} (4);
\node at (0,4) {\footnotesize $X_{SS}$};
\end{scope}

 \begin{scope}[xshift=-0.5cm,yshift=1cm]    
 \node[main node] (2) at (0,1) {{\tiny\bf 2}};
 \node[main node] (3) at (0,2) {{\tiny\bf 1}};
 \node[main node] (4) at (0,3) {{\tiny\bf 0}};
 \path[every node/.style={font=\sffamily\small}]
    (2) edge[bend right=5] node {} (3)
    (3) edge[bend right=5]  node {} (4);
\node at (0,5) {\footnotesize $X_2$};
\end{scope}

 \begin{scope}[xshift=-6.5cm,yshift=-3cm]    
 \node[main node] (1) at (0,0) {{\tiny\bf 1}};
 \node[main node] (2) at (1.1*3,0) {{\tiny\bf 1}};
 \node[main node] (3) at (1.1*1.5,1.1*2.6) {{\tiny\bf 1}};
 \node[main node] (4) at (1.1*.75,1.1*.43) {{\tiny\bf  3}};
 \node[main node] (5) at (1.1*2.25,1.1*.43) {{\tiny\bf 3}};
 \node[main node] (6) at (1.1*1.5,1.1*1.73) {{\tiny\bf 3}};
 \path[every node/.style={font=\sffamily\small}]
    (1) edge[]  node {} (4)
    (1) edge[bend right=16] node {} (4)
    (1) edge[bend left=16] node {} (4)
    (2) edge[]  node {} (5)
    (2) edge[bend right=16] node {} (5)
    (2) edge[bend left=16] node {} (5)
    (4) edge[]  node {} (6)
    (4) edge[bend right=16] node {} (6)
    (4) edge[bend left=16] node {} (6)
    (6) edge[]  node {} (5)
    (6) edge[bend right=16] node {} (5)
    (6) edge[bend left=16] node {} (5)
    (4) edge[]  node {} (5)
    (4) edge[bend right=16] node {} (5)
    (4) edge[bend left=16] node {} (5)
    (3) edge[]  node {} (6)
    (3) edge[bend right=16] node {} (6)
    (3) edge[bend left=16] node {} (6);
\node at (2,4) {\footnotesize $X_{SSS}$};
\end{scope}

 \begin{scope}[xshift=-1.8cm,yshift=-3cm]
 \node[main node] (1) at (.66,0) {{\tiny\bf 6}};
 \node[main node] (2) at (2,0) {{\tiny\bf 6}};
 \node[main node] (3) at (2.67,1.2) {{\tiny\bf 1}};
 \node[main node] (4) at (2,2.4) {{\tiny\bf  2}};
 \node[main node] (5) at (.66,2.4) {{\tiny\bf 2}};
 \node[main node] (6) at (0,1.2) {{\tiny\bf 1}};
 \path[every node/.style={font=\sffamily\small}]
    (1) edge[bend right=8] node {} (2)
    (1) edge[bend right=24] node {} (2)
    (1) edge[bend right=40] node {} (2)
    (1) edge[bend left=8] node {} (2)
    (1) edge[bend left=24] node {} (2)
    (1) edge[bend left=40] node {} (2)
    (2) edge[]  node {} (3)
    (2) edge[bend right=16] node {} (3)
    (2) edge[bend left=16] node {} (3)
    (3) edge[]  node {} (4)
    (4) edge[bend right=8] node {} (5)
    (4) edge[bend left=8] node {} (5)
    (5) edge[]  node {} (6)
    (6) edge[]  node {} (1)
    (6) edge[bend right=16] node {} (1)
    (6) edge[bend left=16] node {} (1);
\node at (1.5,4) {\footnotesize $X_{S2}$};
 \end{scope}
 
 \begin{scope}[xshift=3.5cm,yshift=-3cm]    
 \node[main node] (1) at (0,0) {{\tiny\bf 1}};
 \node[main node] (2) at (0,1) {{\tiny\bf 2}};
 \node[main node] (3) at (-1,2) {{\tiny\bf 2}};
 \node[main node] (4) at (1,2) {{\tiny\bf 2}};
 \path[every node/.style={font=\sffamily\small}]
    (1) edge[]  node {} (2)
    (2) edge[]  node {} (3)
    (3) edge[]  node {} (4)
    (2) edge[]  node {} (4);
\node at (0,4) {\footnotesize $X_3$};    
\end{scope}

 \begin{scope}[xshift=6cm,yshift=-3cm]    
 \node[main node] (1) at (0,0) {{\tiny\bf 3}};
 \node[main node] (2) at (0,1) {{\tiny\bf 2}};
 \node[main node] (3) at (0,2) {{\tiny\bf 2}};
 \node[main node] (4) at (0,3) {{\tiny\bf 1}};
 \path[every node/.style={font=\sffamily\small}]
    (1) edge[]  node {} (2)
    (1) edge[bend right=15] node {} (2)
    (1) edge[bend left=15] node {} (2)
    (2) edge[]  node {} (3)
    (3) edge[]  node {} (4);
\node at (0,4) {\footnotesize $X_S$};    
\end{scope}

\begin{scope}[xshift=-6.5cm,yshift=-8cm,scale=0.85]
 \node[main node] (1) at (1,0) {{\tiny\bf 4}};
 \node[main node] (2) at (3,0) {{\tiny\bf 4}};
 \node[main node] (3) at (4,1.8) {{\tiny\bf 4}};
 \node[main node] (4) at (3,3.6) {{\tiny\bf 4}};
 \node[main node] (5) at (1,3.6) {{\tiny\bf 4}};
 \node[main node] (6) at (0,1.8) {{\tiny\bf 4}};
 \node[main node] (7) at (2,.6) {{\tiny\bf 1}};
 \node[main node] (8) at (3,2.4) {{\tiny\bf 1}};
  \node[main node] (9) at (1,2.4) {{\tiny\bf 1}};
   \node[main node] (10) at (2,1.8) {{\tiny\bf 1}};
 \path[every node/.style={font=\sffamily\small}]
    (1) edge[bend right=30] node {} (2)
    (1) edge[] node {} (2)
    (1) edge[bend right=10] node {} (2)
    (1) edge[bend right=20] node {} (2)    
    (3) edge[] node {} (4)
    (3) edge[bend right=10] node {} (4)
    (3) edge[bend right=20] node {} (4)
    (3) edge[bend right=30] node {} (4)  
    (5) edge[] node {} (6)
    (5) edge[bend right=10] node {} (6)
    (5) edge[bend right=20] node {} (6)
    (5) edge[bend right=30] node {} (6)
    
    (1) edge[double distance=1.5pt]  node {} (10)
    (3) edge[double distance=1.5pt]  node {} (10)
    (5) edge[double distance=1.5pt]  node {} (10)
    (2) edge[double distance=1.5pt]  node {} (7)
    
    (2) edge[double distance=1.5pt]  node {} (8)
    (4) edge[double distance=1.5pt]  node {} (8)
    
    (4) edge[double distance=1.5pt]  node {} (9)
  
    (6) edge[double distance=1.5pt]  node {} (9)
    (3) edge[double distance=1.5pt]  node {} (7)
    (5) edge[double distance=1.5pt]  node {} (8)
   
    (6) edge[double distance=1.5pt]  node {} (7) 
    (1) edge[double distance=1.5pt]  node {} (9);
\node at (2.1,4.7) {\footnotesize $X_{43}$};    
 \end{scope}

 \begin{scope}[xshift=-0.4cm,yshift=-5cm]

  \node[main node] (1) at (0,0) {{\tiny\bf  1}};
  \node[main node] (2) at (-1,-1) {{\tiny\bf 2}};
  \node[main node] (3) at (0,-1) {{\tiny\bf 4}};
  \node[main node] (4) at (1,-1) {{\tiny\bf 2}};
  \node[main node] (5) at (-1,-2) {{\tiny\bf 2}};
  \node[main node] (6) at (0,-2) {{\tiny\bf 4}};
  \node[main node] (7) at (1,-2) {{\tiny\bf 2}};
  \node[main node] (8) at (0,-3) {{\tiny\bf  1}};
  \path[every node/.style={font=\sffamily\small}]
    (1) edge node {} (2)
    (1) edge[bend right=10] node {} (3)
    (1) edge[bend left=10] node {} (3)
          edge node {} (4)
    (2) edge[bend right=10] node {} (5)
    (2) edge[bend left=10] node {} (5)

    (3) edge[bend right=10] node {} (6)
    (3) edge[bend left=10] node {} (6)
    (3) edge[bend right=25] node {} (6)
    (3) edge[bend left=25] node {} (6)
    (4) edge[bend right=10] node {} (7)
    (4) edge[bend left=10] node {} (7)

    (8) edge node {} (5)
    (8) edge[bend right=10] node {} (6)
    (8) edge[bend left=10] node {} (6)
     (8)  edge node {} (7);
\node at (0,1) {\footnotesize $X_{21}$};
\end{scope}

 \begin{scope}[xshift=3cm,yshift=-5cm]
 \node[main node] (1) at (0,0) {{\tiny\bf 1}};
 \node[main node] (2) at (0,-1) {{\tiny\bf 4}};
 \node[main node] (3) at (0,-2) {{\tiny\bf 4}};
 \node[main node] (4) at (0,-3) {{\tiny\bf  1}};
 \node[main node] (5) at (1,-0.5) {{\tiny\bf 2}};
 \node[main node] (6) at (1,-1.5) {{\tiny\bf 4}};
 \node[main node] (7) at (1,-2.5) {{\tiny\bf 2}};
 \path[every node/.style={font=\sffamily\small}]
    (1) edge[bend right=10] node {} (2)
    (1) edge[bend left=10] node {} (2)
    (1) edge node {} (5)
    (2) edge[bend right=10] node {} (3)
    (2) edge[bend left=10] node {} (3)
    (2) edge[bend right=25] node {} (3)
    (2) edge[bend left=25] node {} (3)
    (3) edge[bend right=10] node {} (4)
    (3) edge[bend left=10] node {} (4)
    (5) edge[bend right=10] node {} (6)
    (5) edge[bend left=10] node {} (6)
    (6) edge[bend right=10] node {} (7)
    (6) edge[bend left=10] node {} (7)
    (4) edge node {} (7);
    \node at (0,1) {\footnotesize $X_{22}$};
 \end{scope}

 \begin{scope}[xshift=6cm,yshift=-5cm]
 \node[main node] (2) at (0,0) {{\tiny\bf 2}};
 \node[main node] (3) at (1,-1) {{\tiny\bf 4}};
 \node[main node] (4) at (0,-1) {{\tiny\bf 2}};
 \node[main node] (5) at (0,-2) {{\tiny\bf 2}};
 \node[main node] (1) at (0,-3) {{\tiny\bf  1}};
 \path[every node/.style={font=\sffamily\small}]
    (2) edge node {} (4)
    (3) edge[bend right=10] node {} (4)
    (3) edge[bend left=10] node {} (4)
    (4) edge node {} (5)
    (1) edge node {} (5);
\node at (0,1) {\footnotesize $X_{11}$};
 \end{scope}

 \begin{scope}[xshift=9.5cm,yshift=-5cm]]
 \node[main node] (1) at (0,0) {{\tiny\bf  2}};
 \node[main node] (2) at (-1,-1) {{\tiny\bf 2}};
 \node[main node] (3) at (1,-1) {{\tiny\bf 2}};
 \node[main node] (4) at (0,-2) {{\tiny\bf 2}};
 \node[main node] (5) at (0,-3) {{\tiny\bf 1}};
 \path[every node/.style={font=\sffamily\small}]
    (1) edge node {} (2)
         edge node {} (3)
    (2) edge node {} (4)
    (3) edge node {} (4)
    (4) edge node {} (5);
    \node at (0,1) {\footnotesize $X_{10}$};
 \end{scope}
\end{tikzpicture}
\caption{Intersection graphs for the effective cones.}
\label{tab:int}
\end{table}
\end{center}
\end{introthm}

The paper is structured as follows. In Section~\ref{sec:setup},
we introduce del Pezzo elliptic fibrations and del Pezzo elliptic
varieties and we define the bilinear form on the Picard group
of such varieties. In Section~\ref{sec:types}, we study 
the geometry of these varieties and in the next section
we use these results in order to classify the Mordell-Weil groups 
of the del Pezzo elliptic fibrations, their vertical classes and sections. 
Section~\ref{sec:cones} contains the description 
of the nef, effective and moving cones of del 
Pezzo elliptic varieties and
moreover in the same section we prove 
Theorem~\ref{cox-mw}. In the last section, we 
provide the Cox rings of the del Pezzo elliptic varieties
whose fibration has finite Mordell-Weil
group and having degree one, two and four (few examples),  
and a lemma about the Cox ring of the blow-up in one point of the 
complete intersection of two quadrics. 

\section{Del Pezzo elliptic varieties}
\label{sec:setup}

Let $Y$ be a del Pezzo variety of dimension $n\geq 3$
such that $-K_Y=(n-1)H$, with $H$ ample and $d:=H^n\leq 4$. 
It is well known (see for instance~\cite{dp}) that 
the picard group
of $Y$ has rank one and it is generated by the class $H$. 
Let us recall the following. If $d=1$ then $Y$ is a smooth 
hypersurface of degree six of the weighted projective space
$\pp(3,2,1,\dots,1)$ and $H$ is the restriction of a degree one 
class of the ambient space.
If $d=2$ then $Y$ is a double cover of $\pp^n$ 
branched along a smooth quartic hypersurface 
and $H$ is the pull-back of a hyperplane of $\pp^n$.
If $d\in\{3,4\}$ then $Y$ is a projectively normal subvariety
of $\pp^{n+d-2}$ and $H$ is the class of a hyperplane
section. 

Let us consider a $n-1$-dimensional sublinear system of
$|H|$, whose base locus $Z$ has dimension zero and 
length $d$. In particular, if $d=1$ we have $Z=V(x_3,\dots,x_{n+2})$,
if $d=2$, $Z$ is preserved by the covering involution and
if $d\in\{3,4\}$, $Z$ spans a linear subspace $\Lambda\subseteq
\pp^{n+d-2}$ of dimension $d-2$.
Let us denote by $\pi_{Z}\colon Y\dashrightarrow\pp^{n-1}$ 
the rational map defined by the given system
and by $\pi\colon X\to\pp^{n-1}$ 
the resolution of the indeterminacy of 
$\pi_Z$. 
The variety $X$ comes with two morphisms:
\[
 \xymatrix{
  X\ar[r]^-\pi\ar[d]_-\sigma & \pp^{n-1}\\
  Y\ar@{-->}[ur]_-{\pi_Z} 
 }
\]
where $\sigma$ is the composition 
of $d$ blowing-ups $\sigma_1,
\dots ,\sigma_d$ at the points $q_1,...,q_d$, respectively.
Moreover, assuming that $\Lambda$ is not contained
in the tangent space of $Y$ at any point of $Z$
when $d=4$, the general fiber of $\pi$ is a smooth genus one
curve, that is $\pi$ is an elliptic fibration.

In what follows, by abuse of notation, we use the same letter 
$H$ to denote the pull-back of $H$ via $\sigma$ while 
we denote by $E_i$ the pull-back of the exceptional
divisor of $\sigma_i$, for $i\in\{1,\dots,d\}$.
Observe that some of the points $q_2,\dots,q_d$
can lie on the exceptional divisor of one of the $\sigma_i$'s.
Therefore $E_i$ can be either a $\pp^{n-1}$ 
or the union of a $\pp^{n-1}$ with some
other components isomorphic to 
the projectivization $\mathbb{F}$ of the vector bundle
$\Osh_{\pp^{n-1}}\oplus\Osh_{\pp^{n-1}}(1)$.
In any case, we can write
\[
 \Pic(X) = \langle H, E_1, \dots,E_d\rangle,
\]
where, with abuse of notation, we are adopting
the same symbols for the divisors and for their
classes. We will also adopt the following notation
\begin{equation}
\label{F}
 F := -\frac{1}{n-1}K_X .
\end{equation}
Observe that $F$ is the pull-back of a hyperplane
section of $\pp^{n-1}$ via $\pi$, so that  $F=
H-\sum_{i=1}^dE_i$.

\begin{remark}\label{tree}
The map $\sigma$ is a composition of blow-ups
at points and we claim that we blow up at most one point on each
exceptional divisor. Indeed, assume by contradiction
that there is a prime divisor $E$ which is the strict transform
of an exceptional divisor blown up at two or more points. 
The preimage $S$ of a general line $\ell$ of $\pp^{n-1}$ via $\pi$
is a rational elliptic surface with nef anticanonical
class which contains a prime divisor $E|_S$ of self-intersection
$<-2$, a contradiction.
\end{remark}

\subsection{A bilinear form on the Picard group}

Let now $X$ be the blow-up of $Y$ at $r$ general
points. Using the above notation for $F$,
we introduce a bilinear form on 
$\Pic(X)$ by setting 
\begin{equation}\label{bil}
 \langle A, B\rangle
 :=
 F^{n-2}\cdot A\cdot B
\end{equation}
for any two divisors $A$ and $B$ on $X$.
Thus the quadratic form $q$ induced by 
the above linear form is hyperbolic
and the matrix with respect to the basis 
$(H,E_1,\dots, E_r)$ is diagonal
with entries $d,-1,\dots,-1$.
Since $\langle F,F\rangle = d-r$,
the sublattice $F^\perp$ is negative
definite if $1<r<d$ and it is negative
semidefinite if $r = d$. In the first case, a basis
consists of the classes $E_1-E_2,\dots,E_{r-1}-E_r$,
while in the second case it consists of the
above classes plus $F$. These are roots lattices
of type $A_{r-1}$ and $\tilde{A}_{d-1}$,
respectively.

When $r=d$ and the linear system $|F|$ on the
blow-up $X$ induces the elliptic fibration $\pi
\colon X\to\pp^{n-1}$, we observe that $F^{n-2}$ is
rationally equivalent to a smooth rational elliptic
surface $S$ which is the preimage via $\pi$ of a line.
Thus we have $\langle A,B\rangle = A|_S\cdot B|_S$,
where the right hand side is the intersection product
in $\Pic(S)$.

\begin{proposition}\label{fimov}
\label{base}
Let $A$ and $B$ be effective divisors of $X$
with $B$ a prime divisor. If $\langle A,B\rangle < 0$
then $B$ is contained in the stable base locus
of $|A|$.
\end{proposition}
\begin{proof}
Let $\ell$ be a general line of $\pp^{n-1}$ and let $S$
be the surface $\pi^{-1}(\ell)$. According
to the definition of the bilinear form we have 
$A|_S\cdot B|_S < 0$. Being $B$ prime and $\ell$
general, the divisor $B|_S$ of $S$ is prime as well.
Thus the linear series $|A_{|S}|$ contains $B|_S$
into its base locus and the same holds for the linear
series $|A|$. Varying $\ell$ we get the claim.
\end{proof}

\section{Types}
\label{sec:types}

In this section we are going to describe 
the possible types of del Pezzo elliptic varieties 
of 
degree $d\leq 4$. 

\subsection{Degree one}
In this case $Y$ is a degree $6$ hypersurface
of the $(n+1)$-dimensional weighted 
projective space $\pp(3,2,1,\dots,1)$. 
After applying a change of coordinates,
we can assume that a defining equation
for $Y$ is
\begin{equation}
\label{type:1}
 x_1^2-x_2^3+x_2f_4+f_6 = 0,
\end{equation}
where $f_t$ is a degree $t$ homogeneous
polynomial in $x_3,\dots,x_{n+2}$. The 
blow-up $\sigma\colon X\to Y$ is centered
at the point $q=(1,1,0,\dots,0)\in Y$ and
the rational map $Y\dashrightarrow \pp^{n-1}$ is 
defined by $(x_1,\dots,x_{n+2})\mapsto (x_3,\dots,x_{n+2})$.

\subsection{Degree two}
In this case $Y$ is a double covering 
$\varphi\colon Y\to \pp^n$ branched
along a smooth quartic hypersurface $S$. 
In order to 
distinguish the different cases that can occur, we
observe that the preimage of a line $\ell$ through 
a point $p:=\varphi(q_1)$ is one of the following:
\begin{equation}
\label{fibres:cover}
 \varphi^{-1}(\ell)
 =
 \begin{cases}
 \text{elliptic curve} & \text{if $|\ell\cap S| = 4$}\\
 \text{rational nodal curve} & \text{if $|\ell\cap S| = 3$}\\
 \text{union of two smooth rational curves} & \text{if $\ell$ is bitangent to $S$.}\\
 \end{cases}
 \end{equation}
Therefore we distinguish three 
different cases depending on the position of $p$
with respect to $S$ and on the dimension of
the variety $B\subseteq\pp^n$ spanned by the
bitangent to $S$ passing through $p$.

{\em Case $1$}.
The point $p$ does not lie on $S$ and 
$B$ is not a hypersurface. In this case the preimage of 
$p$ in the double covering $Y\to\pp^n$ consists 
of two distinct points $q_1$ and $q_2$. 
We denote by $X_{11}$ the variety that we
obtain by blowing up these two points.

{\em Case $2$}.
The point $p$ does not lie on $S$ and
$B$ is a hypersurface. In this case after a linear change 
of coordinates we can assume that $p=(0,\dots,0,1)$.
An equation for $Y$ has the following form
\begin{equation}
\label{type:ss}
 x_{n+2}^2 = g + h^2,
\end{equation}
where $g\in\cc[x_1,\dots,x_n]$ is a homogeneous polynomial of
degree four such that $V(g)$ is the cone spanned by the bitangents
through $p$, while $h\in \cc[x_1,\dots,x_{n+1}]$ is a homogeneous 
polynomial of degree two such that $h(p)\neq 0$ and $S=V(g+h^2)$.
We denote by $X_{SS}$ the variety 
obtained by blowing up the two distinct points
$q_1$ and $q_2$ in the pre-image of $p$.
 
{\em Case $3$}.
The point $p$ lies on $S$. In this situation 
$B$ cannot be a hypersurface since otherwise $S$
would be singular. In order to get an elliptic fibration
we need to blow up the point $q_1:=\varphi^{-1}(p)$
and the point on the exceptional divisor 
which is invariant with respect to the lifted involution.
We denote by $X_2$ the variety that we obtain
after the blowing-ups. 
In this case an equation for $Y$ has the following form
\begin{equation}
\label{type:2}
 x_{n+2}^2 = x_nx_{n+1}^3 + f.
\end{equation}
where $f\in\cc[x_1,\dots,x_{n+1}]$ is a homogeneous polynomial
of degree four which does not contain monomials of
degree $\geq 3$ in the variable $x_{n+1}$.
The point $q_1$ has coordinates $(0,\dots,0,1,0)$ and
the tangent space to $S$ at $p$ is $V(x_{n})$.

\subsection{Degree four} 
Let us first collect some fact about 
smooth complete intersections of two hyperquadrics
$Y:=Q\cap Q'\subseteq\pp^{n+2}$, for $n\geq 3$.
Observe that any quadric in the pencil 
$\Lambda$ generated by $Q$ and $Q'$ has rank 
at least $n+2$, since otherwise $Y$ would not be smooth,
and there are $n+3$ singular quadrics in the pencil, counting muliplicities. 
We claim that there are exactly $n+3$ quadrics of rank $n+2$ 
and their vertices  are in general position in $\pp^{n+2}$.
Indeed, let us suppose that either there are less than $n+3$
vertices or that they are not in general position.

In the former case the pencil of quadrics is tangent
to the discriminant hypersurface at some point.
Without loss of generality we can assume $Q$ 
to be a cone of vertex $p = (1,0,\dots,0)$ in 
diagonal form $g$ and the pencil $g+tg'$ is
tangent to the discriminant hypersurface 
at $t=0$. If the Hessian matrix of $g$ is $M$
and that of $g'$ is $M'$, the above tangency
condition is equivalent to the vanishing of the
following derivative
\[
 \frac{d}{dt}{\rm Det}(M+tM')|_{t=0}.
\]
Expanding the above derivative and using
the fact that $M$ is diagonal we see that the above
is equivalent to $m'_{11}=0$, that is $p\in Q'$.
This is not possible since it contradicts the smoothness
of $Y$.

In the latter case there exists a hyperplane $H\subseteq\pp^{n+2}$ 
containing all the vertices 
and if we restrict $\Lambda$ to $H$ we obtain a pencil 
$\Lambda_H$ of quadrics in $\pp^{n+1}$, containing at least 
$n+3$ singular quadrics (counting multiplicities). Hence all the 
quadrics of $\Lambda_H$ must be singular and by Bertini's theorem
their vertices are contained in the base locus of $\Lambda_H$.
This implies that all the vertices of these cones are in $Y$ 
and this is a contradiction since they give singular points
of $Y$.

In what follows we will denote by
$Q_1,\dots,Q_{n+3}$ the singular quadrics 
and by $p_1,\dots,p_{n+3}$ the corresponding 
vertices. By the above discussion, we can assume
that $p_i$ is the $i$-th fundamental point of $\pp^{n+2}$
for $i=1,\dots,n+3$, so that $Q_1$ and $Q_2$ are 
defined by diagonal forms. Moreover, after possibly rescaling the variables, 
we can assume the quadrics
to be defined by the following polynomials 
\begin{equation}
\label{equ:Y}
 x_2^2-x_3^2+x_4^2+\dots +x_{n+3}^2
 \qquad
 x_1^2-x_3^2+\alpha_4x_4^2+\dots + \alpha_{n+3}x_{n+3}^2
\end{equation}
respectively, where the coefficients $\alpha_i$ are distinct and
not in $\{0,1\}$.
Let us now prove the following result that
will be useful in the next section.

\begin{proposition}\label{inf}
Let $q_1$ and $q_2$ be two points of $Y$, possibly
infinitely closed. Then
the conics of $Y$ through these 
two points span a hypersurface of $Y$ if
and only if the line $\langle q_1,q_2\rangle$ 
passes through one of the vertices $p_i$.
\end{proposition}
\begin{proof}
If $p_i$ lies on the line $\langle q_1,q_2\rangle$,
then this line is a generatrix of the cone $Q_i$.
We can write $Y=Q_i\cap Q$, where $Q$ is
any other quadric of the pencil. We conclude 
observing that there exists an $(n-2)$-dimensional 
family of planes of $Q_i$ containing a generatrix 
and each of them intersects $Q$ along a conic
through the two fixed points.

Let us suppose now that the conics through 
$q_1$ and $q_2$ span a hypersurface $S$,
i.e. there exists an $(n-2)$-dimensional family
of such conics.
Observe that when we have a conic $C$ contained in
$Y$ then the plane $\pi_C$ of the conic must be contained
in one of the quadrics of the pencil $\Lambda$
(since the generic quadric of the pencil cuts $\pi_C$
along the curve $C$, imposing to pass through
one point of $\pi_C$ not lying on $C$ we get the whole plane).
Therefore, under our hypotheses, we must have
a quadric of the pencil containing an
 $(n-2)$-dimensional family of planes sharing the 
 line $\langle q_1,q_2\rangle$. Hence this quadric
 is a cone with vertex on that line.
 \end{proof}

Let us fix now a plane $\Pi\subseteq\pp^{n+2}$
and let us analyze the different types of
del Pezzo elliptic varieties of degree four.
By Proposition~\ref{inf} the type depends not
only on the number of points we blow up
but also on the number of vertices $p_i$ 
contained in the plane $\Pi$. Hence we are going to 
use the symbol $X_{kl}$ to denote the variety that
we obtain by choosing a plane $\Pi$ intersecting 
$Y$ in $k$ distinct points and containing $l$ vertices.

Let us spend few words about the geometry
of this construction and about the possible
values of $k$ and $l$ for $X_{kl}$.
We remark that we can write 
\[
 \Pi\cap Y=C\cap C'
\]
where $C:=Q\cap \Pi$ and $C':=Q'\cap \Pi$
are two plane conics. 
We discuss four cases.

{\rm Case 1}. 
If $\Pi$ contains no vertices, then we have two smooth conics, 
whose intersection consists of $k$ distinct points, for
$k\in\{1,2,3,4\}$ and hence we get the types 
$X_{40},\,X_{30},\,X_{20},\,X_{10}$.
Observe that when $k=2$ we have
two possibilities: either $C$ and $C'$ are tangent 
at their two intersection points $q_1$ and $q_2$, or 
they intersect transversally at $q_2$ and with 
multiplicity three at $q_1$.

{\rm Case 2}. 
If $\Pi$ contains one vertex $p_i$, then 
we can suppose that $C$ is a smooth
conic while $C':=\Pi \cap Q_i$ has (at least) 
a singular point at the vertex $p_i\in\Pi$.
The intersection of $C$ and $C'$ consists
of $k$ points, for $k\in\{1,2,3,4\}$ and we obtain
the types $X_{41},\,X_{31},\,X_{21},\,X_{11}$.
As before, when $k=2$ we have two 
possibilities. 
Either $C'$ is the union of two distinct lines and each
of them is tangent to the conic $C$, or $C'$ is a double line 
(which means that $\Pi$ is tangent to $Q_i$) 
intersecting $C$ in two distinct points.

{\rm Case 3}. 
If $\Pi$ contains two vertices, then
we can suppose that both $C$ and $C'$ are
singular and they can not intersect at the
vertices so that $k$ can be either $1,2$ or $4$.
Moreover, when $k=1$ we deduce that 
the plane $\Pi$ is contained in the tangent space to $Y$ at the 
only intersection point $q_1$. 
We are not going to consider this case since 
it does not give an elliptic fibration, being all
the fibers singular rational curves.
Hence we have only the two types $X_{42}$ and $X_{22}$.

{\rm Case 4}. 
Finally, observe that if $\Pi$ contains three vertices
then it is fixed and it can intersect $Y$ only at four
distinct points (otherwise $Y$ would be singular), 
giving case $X_{43}$.

\begin{remark}
We provide here an example of defining equations for $\Pi$ 
for each of the following five types:
\begin{align*}
 X_{43}:\hspace{5mm} & \Pi = V(x_4,x_5,\dots,x_{n+3})\\
 X_{22}:\hspace{5mm} & \Pi = V(x_3-x_4,x_5,\dots,x_{n+3})\\
 X_{21}:\hspace{5mm} & \Pi = V(\sqrt{\alpha_4+\alpha_5}\cdot x_2-\sqrt{\alpha_4+\alpha_5-2}\cdot x_3,x_4-x_5,\dots,x_{n+3})\\
 X_{11}:\hspace{5mm} & \Pi = V(x_1-\alpha_4\,x_2+(\alpha_4-1)\,x_3,x_5,\dots,x_{n+3})\\
 X_{10}:\hspace{5mm} & \Pi = V(2x_1-(\alpha_4+\alpha_5)x_2+(\alpha_4+\alpha_5-2)x_3,
 x_4-x_5,x_6,\dots,x_{n+3}),
\end{align*}
where $\alpha_4+\alpha_5\neq 0,2$ in cases $X_{21}$ and $X_{10}$.
\end{remark}

\begin{remark}
We recall that if $Y=Q\cap Q'\subseteq\pp^{n+2}$ and $n\geq 3$,
then through any point of $Y$ we have at least one 
line of $Y$. So let us fix a point $q_i\in Y$ and a line $\ell$
of $Y$, passing through this point, and let us describe the
fiber of $\pi\colon X\to \pp^{n-1}$ containing the strict 
transform of that line. The image of this fiber inside $Y$
is the curve obtained by intersecting $Y$ with the $\pp^3$ 
spanned by the plane $\Pi$ and the line $\ell$. This 
can also be described as the base locus of the pencil of quadric
surfaces obtained by restricting $\Lambda$ to the $\pp^3$
that we are considering.
Observe that any time we have a vertex $p_i$
in $\Pi$, the intersection of $Q_i$ with the $\pp^3$ is
a quadric cone containing a line not passing through 
$p_i$. Hence it must be the union of two planes
intersecting along a line passing through $p_i$.
Therefore, in Case 1 the image of the fiber inside $Y$ is obtained 
by intersecting two smooth quadric surfaces sharing a 
line and hence it is the union of that line and a rational 
normal cubic, intersecting in two points.
In Case 2 the base locus is the
intersection of a smooth quadric with a reducible
one and then it is the union of two lines and a
smooth conic. 
In Case 3 the base locus is the intersection of two reducible 
quadrics and hence it consists of four lines. 
Finally, in Case 4 we have the base locus of a pencil
containing three reducible quadrics.
Thus, after a possible 
renaming of the coordinates,
the pencil has the form $(x_1^2-x_2^2)+t(x_2^2-x_3^2)$.
All the quadrics in this pencil are
singular at the point $p=(1,1,1,1)$ and the
base locus of the pencil consists of four lines
intersecting at the point $p$. 
We remark that in this last case the corresponding 
fiber in $X$
is the union of four rational components
passing through one point and hence it is a type 
that does not appear in the Kodaira's list of singular fibers for 
elliptic surfaces. 
\end{remark}

\section{Mordell-Weil groups}
\label{sec:mw}

The main result of this section is the proof of Theorem~\ref{thm:mw}
but we postpone it to the end of the section and we begin by studying the 
{\em vertical divisors} of all the del Pezzo elliptic 
fibrations of degree $d\leq 4$, that is
divisors $D$ such that $\pi(D)$ is a hypersurface
of $\pp^{n-1}$.
If $d=1$, 
then the only vertical class is $F$ since the rank of the subgroup of
vertical divisors equals $\rk\Pic(X) - 1 = 1$.

When $d=2$, 
recall that there is a double covering 
$\varphi\colon Y\to\pp^n$ branched along a smooth quartic 
hypersurface $S$ and $\sigma\colon X\to Y$ 
is the blow-up of $Y$ at two points $q_1,q_2$
exchanged by the covering involution.
By~\eqref{fibres:cover}, if $D$ is a prime proper vertical
divisor of $X$ whose class is not a multiple of $F$,
then either $D$ is contained in the pull-back of an 
exceptional divisor of $\sigma$, or $\varphi(D)$ is 
covered by bitangent lines to $S$.
Therefore in case $X_{11}$ there are no proper vertical divisors.

In case $X_{SS}$ we have the two vertical classes
$2H-4E_1$, $2H-4E_2$, and assuming that $Y$ 
has the equation~\eqref{type:ss}, they are the
classes of the strict transforms
of $V(x_{n+2}-h)$ and $V(x_{n+2}+h)$, respectively.

Finally, in case $X_2$, $E_1-E_2$ and $H-2E_1$ are 
the only prime proper vertical divisors.

The case $d=3$ 
has already been studied in~\cite{hltu} and we refer to that paper for the 
classification of prime proper vertical divisors.

For $d=4$, 
if $D$ is a prime proper vertical
divisor, then $D$ is strictly contained in the support 
of $\pi^*\pi(D)$. Let $\gamma$ be a general fiber of
$\pi$ over a point $q\in\pi(D)$ and let us denote by $C$
the image $\sigma(\gamma)$ in $Y$. Then either
(j) $C$ is an irreducible rational curve or (jj) it contains
lines and/or conics.

In case (j), $C$ is singular
at one of the points $q_i\in\Pi\cap Y$ and
the union of these curves gives a prime
proper divisor having class $H-2E_i-E_j-E_k$.
In order to obtain the class of a fiber we have
to add some prime proper vertical exceptional divisors 
of the form $E_i-E_j$. 

In case (jj), observe that 
by~\cite{cps} through any point of $Y$ there is 
only a  $(n-3)$-dimensional family of lines and 
hence they can not fill up a divisor. Therefore
the curve $C$ must contain a conic through
two points $q_i$ and $q_j$ of $\Pi \cap Y$, possible infinitely 
near. By Proposition~\ref{inf} the line
$\langle q_i,q_j\rangle$ passes through  
one vertex $p_k$ and hence $p_k\in\Pi$. 
In this case the class of one of the irreducible 
components of $\pi^*\pi(D)$ is of the form
$H-2E_i-2E_j$.
For instance, in case $X_{31}$ we can write
$Y=Q_1\cap Q$, where $Q_1$ is a cone
with vertex $p_1\in\Pi$ and $Q$ is a smooth
quadric. Furthermore, $Q$ intersects $\Pi$ along 
a smooth conic $C$ while $Q_1\cap \Pi$ is the 
union of two generatrices and one of them is
tangent to $C$ at $q_1$ while the other one 
intersects $C$ in $q_3$ and $q_4$.
Therefore we have the vertical class $H-2E_1-2E_2$ 
corresponding to the conics through $q_1$ and 
whose tangent line at $q_1$ is the line 
$\langle q_1,p_1\rangle$ and the class $H-2E_3-2E_4$
corresponding to the conics through $q_3$ and $q_4$.
Observe that the sum of these classes gives twice 
a fiber. Moreover, we also have 
the vertical class $E_1-E_2$ sitting inside the exceptional locus
and the class $H-2E_1-E_3-E_4$ which is spanned by 
the union of the strict transforms of the singular rational 
quartic curves of $Y$ obtained intersecting it with
a hyperplane  tangent to $Y$ at $q_1$.

We summarise the above observations in the following

\begin{proposition}
\label{pro:vertical}
Let $\pi\colon X\to \pp^{n-1}$ be a del Pezzo elliptic 
fibration of degree $d\leq 4$ and dimension $n\geq 3$.
Then for each type the sections and the vertical divisors
are as follows:

{\footnotesize 
\begin{longtable}{|c|c|l|l|}
\hline
 Degree & Type & Sections & Proper prime vertical divisors\\[2pt]
\hline
$1$ &
$X_{1}$ &
$E_1$
& \\[2pt]
\hline
$2$ &
$X_{11}$ &
$E_1,E_2$
& \\[2pt]
& 
$X_{SS}$ &
$E_1,E_2$
& $2H-4E_1,2H-4E_2$\\[2pt]
& 
$X_{2}$ &
$E_2$
& $E_1-E_2, H-2E_1$ \\[2pt]
\hline
$3$ &
$X_{111}$ &
$E_1,E_2,E_3$
& \\[2pt]
& 
$X_{S11}$ &
$E_1,E_2,E_3$
& $H-3E_1, 2H-3E_2-3E_3$\\[2pt]
& 
$X_{SSS}$ &
$E_1,E_2,E_3$
& $H-3E_1, H-3E_2, H-3E_3$\\[2pt]
& 
$X_{12}$ &
$E_1,E_3$
& $E_2-E_3, H-E_1-2E_2$\\[2pt]
& 
$X_{S2}$ &
$E_1,E_3$
& $H-3E_1, 2H-3E_2-3E_3, E_2-E_3, H-E_1-2E_2$\\[2pt]
& 
$X_{3}$ &
$E_3$
& $E_1-E_2, E_2-E_3, H-2E_1-E_2$ \\[2pt]
& 
$X_{S}$ &
$E_3$
& $E_1-E_2, E_2-E_3,H-2E_1-E_2,H-3E_1,2H-3E_2-3E_3$\\[2pt]
\hline
$4$ & 
$X_{40}$ &  
$E_1,E_2,E_3,E_4$
&  \\[2pt]
& 
$X_{41}$ & 
$E_1,E_2,E_3,E_4$
& $H-2E_1-2E_2, H-2E_3-2E_4$ \\[2pt]
& 
$X_{42}$ & 
$E_1,E_2,E_3,E_4$
& $H-2E_1-2E_2,H-2E_3-2E_4$\\[2pt]
&&& $H-2E_1-2E_3, H-2E_2-2E_4$  \\[2pt]
& 
$X_{43}$ & 
$E_1,E_2,E_3,E_4$
& $H-2E_i-2E_j,\ 1\leq i<j\leq 4$ \\[2pt]
& 
$X_{30}$ &  
$E_2,E_3,E_4$
& $E_1-E_2, H-2E_1-E_3-E_4$ \\[2pt]
& 
$X_{31}$ & 
$E_2,E_3,E_4$
& $H-2E_1-2E_2,  H-2E_3-2E_4$\\[2pt]
&&& $E_1-E_2, H-2E_1-E_3-E_4$  \\[2pt]
& 
$X_{20}$ &  
$E_3,E_4$
&  $E_1-E_3, H-2E_1-E_2-E_4, E_2-E_4,H-E_1-2E_2-E_3$ \\[2pt]
&&
$E_3,E_4$
& $E_1-E_3, E_3-E_4, H-2E_1-E_2-E_3$  \\[2pt]
& 
$X_{21}$ & 
$E_3,E_4$
& $E_1-E_3, E_2-E_4,H-2E_1-2E_3, H-2E_2-2E_4$\\[2pt]
&&& $H-2E_1-E_2-E_4, H-E_1-2E_2-E_3$ \\[2pt]
& &
$E_3,E_4$
& $E_1-E_3, E_2-E_4, H-2E_1-E_2-E_4, H-E_1-2E_2-E_3$ \\[2pt]
& 
$X_{22}$ & 
$E_3,E_4$
& $H-2E_1-2E_3, H-2E_2-2E_4$\\[2pt]
&&& $E_1-E_3, E_2-E_4, H-2E_1-2E_2$\\[2pt]
& 
$X_{10}$ &  
$E_4$
& $E_1-E_2, E_2-E_3,E_3-E_4, H-2E_1-E_2-E_3$\\[2pt]
& 
$X_{11}$ & 
$E_4$
& $E_1-E_2, E_2-E_3, E_3-E_4, H-2E_1-2E_2$ \\[2pt]
\hline

\caption{Sections and vertical classes of del Pezzo elliptic fibrations with $d\leq 4$.}
\label{sect:verti}
\end{longtable}
}

\end{proposition}

\begin{proof}[Proof of Theorem~\ref{thm:mw}]
Recall that the Mordell-Weil group
of the elliptic fibration $\pi$ is the
group of rational sections of $\pi$ or, equivalently,
the group of $K=\cc(\pp^{n-1})$-rational
points $X_\eta(K)$ of the generic fiber $X_\eta$ of 
$\pi$ once we choose one of such points 
$O$ as an origin for the group law.
Let $\mathscr{T}$ be the subgroup of
$\Pic(X)$ generated by the classes of the
vertical divisors and by the class of the 
section $O$. There is an exact
sequence~\cite{Wa}*{Section~3.3}:
\begin{equation}\label{mw}
 \xymatrix{
  0\ar[r] &
  \mathscr{T}\ar[r] &
  \Pic(X)\ar[r] &
  X_\eta(K)\ar[r] &
  0.
 }
\end{equation}
In degree $d$, the Picard group of $X$ is free of rank 
$d+1$, generated by the classes $H,E_1,\dots,E_d$.
Observe that if $F$ is defined as in~\eqref{F}, then 
$\langle F,E_d\rangle\subseteq\mathscr T$
holds and by Proposition~\ref{pro:vertical} and the 
sequence~\eqref{mw} we get the statement.
\end{proof}

\section{Cones}
\label{sec:cones}

The aim of this section is to provide a description
of the nef, effective and moving cone of del Pezzo
elliptic varieties. Moreover, we discuss the Mori chamber 
decomposition of the moving cones and we use this 
decomposition in order to prove Theorem~\ref{cox-mw}.

\subsection{The nef cones}
Given a subset $I$ of $\{1,\dots,d\}$, in what
follows we denote
by $F_I$ the divisor $H-\sum_{i\in I}E_i$.

\begin{theorem}
\label{cones}
Let $\pi\colon X\to\pp^{n-1}$ be a del Pezzo
elliptic fibration with $n\geq 3$. 
Then the extremal rays of the nef cone $\Nef(X)$ 
are all the $F_I$ such that $I\subseteq\{1,\dots,d\}$ and
$\langle F_I,V\rangle\geq 0$ for each exceptional vertical 
class $V$.
\end{theorem}
\begin{proof}
Let us consider the subcone $\mathcal C$ 
of the Mori cone of $X$ generated by the following
classes: 
\begin{itemize}
\item $e_i$ such that $E_i$ is a section;
\item $e_i-e_j$ such that $E_i-E_j$ is a prime vertical
divisor;
\item $h-e_i$ for each $q_i\in Y$. 
\end{itemize}
Let $D:=\alpha H-\sum m_iE_i$ be a class in 
the dual $\mathcal C^*$. Then we have the following
inequalities: 
$m_i\geq 0\ \forall i$, $m_i\geq m_j$ if the
point $q_j$ lies on the exceptional divisor of
the blowing-up at $q_i$ and finally $\alpha\geq
m_i\ \forall i$. Let us write $\{m_1,\dots,m_d\}=
\{\mu_1,\dots,\mu_r\}$, where $r\leq d$ and
$0=\mu_0 \leq \mu_1 < \dots < \mu_r$,
and let us denote
by $I_i:=\{j\mid m_j\geq \mu_i\}$, for each $i=1,\dots,r$. 
Then we can write
\[
D=(\alpha-\mu_r)H+\sum_{i=1}^r(\mu_i-\mu_{i-1})F_{I_i},
\]
where the $F_{I_i}$ are
nef and their product with any effective $E_j-E_k$ is
non negative. In order to conclude the proof we need to
show that these $F_I$ are extremal rays of the nef cone.

Let us first suppose that $X$ is obtained by blowing up 
$d$ distinct points on $Y$. In this case, we have to consider
all the $F_I$ as $I$ varies in the subsets of $\{1,\dots,d\}$, 
and by induction on $d$ it can be proved that they
are vertices of a $d$-dimensional hyper-cube. In particular,
they are extremal rays of the cone they generate.

In addition, we can also infer that 
no $F_I$ lies in the convex hull of the remaining and hence the general
case follows.
\end{proof}

\subsection{The effective and moving cones}
We now restrict our attention to del Pezzo elliptic
fibrations of degree $d\leq 4$ and having
finite Mordell-Weil group, proving Theorem~\ref{eff-mov-cones}.

\begin{proof}[Proof of Theorem~\ref{eff-mov-cones}]
Let us consider, for each 
del Pezzo elliptic variety $X$ the cone $\mathcal M$ of $\Pic_\qq(X)$
generated by the vertical classes and the sections of $\pi$. 
Let $\varrho_1,\dots,\varrho_n$ be the
extremal rays of $\mathcal M$.
We have the following inclusions
\begin{equation}\label{inclusions}
 \bigcap_{i=1}^n\,
 {\rm cone}(\varrho_1,\dots,\stackrel{\vee}{\varrho_i},\dots,\varrho_n)
\subseteq
\Mov(X)\subseteq\Eff(X)^\vee\subseteq\mathcal M^\vee,
\end{equation}
where the first inclusion is due to the fact that each
$\rho_i$ is generated by a prime divisor,
the second one is a consequence of  Proposition~\ref{fimov}
and the last one follows from $\mathcal M\subseteq\Eff(X)$.
The proof goes as follows. If the degree $d$ is at most three, then
the Cox ring is known (Theorem~\ref{cox:12}
for degree one or two and~\cite{hltu} for degree three)
and a direct computation shows
that the rays of $\mathcal M^\vee$ are movable. 
When $d=4$, 
observe that if $X$ is of type $X_{43}, X_{22}, X_{21}$, then the first cone and the last one
in~\eqref{inclusions} are equal and the two assertions
of the theorem follow.
In the remaining cases, we are going to check that
the rays of $\mathcal M^\vee$ are movable.

If $X$ is of type $X_{11}$, then the only class we have to check is 
$H-2E_1$ (all the other rays of $\mathcal M^\vee$ are
indeed nef classes). We are going to see that the base locus
of the linear system $|H-2E_1|$ has codimension two.
Indeed, this linear system corresponds on $Y$ to the
linear system of hyperplane sections
containing the tangent space
at $q_1$, whose base locus is the 
union of the lines passing through
$q_1$. When we blow up $q_1$, the strict
transforms of these lines intersect $E_1$
along a subvariety of codimension two. Observe 
that the second point $q_2$ that we blow up
do not lie on this subvariety, since
otherwise the plane $\Pi$ would intersect
$Y$ along a line. Then the base locus of
$|H-2E_1|$ can not be divisorial.

In case $X_{10}$ the only classes we have to check
are $H-2E_1$ and $3H-4E_1-4E_2$. The first one
can be done as in case $X_{11}$ while the
second one can be obtained as
the image of $H$ 
via the Geiser involution described in Subsection~\ref{flop}, 
and hence it is movable.
\end{proof}
As a consequence of Theorem~\ref{eff-mov-cones}, 
if $X$ is a del Pezzo elliptic variety of degree $d\leq 4$ 
such that the Mordell-Weil group of $\pi\colon X\to\pp^{n-1}$
is finite, then the effective cone $\Eff(X)$ can be read from 
Table~\ref{sect:verti}.
The graphs of the quadratic form on the
primitive generators of the extremal rays 
of $\Eff(X)$ are listed in Table~\ref{tab:int}.

\vspace{5mm}

Let us consider an example in which the 
Mordell-Weil group of the fibration is not finite
and the moving cone is the union of infinitely many 
chambers.

When the elliptic fibration has degree $d=2$ and type $X_{11}$, 
we have seen that the Mordell-Weil group is $\langle \sigma\rangle
\cong \zz$. The action of $\sigma$ on the Picard 
group of $X$, with respect to the basis 
$B:=(H-E_1-E_2,\, E_2-E_1,\, E_1)$,
is given by the following matrix
\[
 \begin{pmatrix}
  1&2&0\\
  0&1&1\\
  0&0&1
 \end{pmatrix}.
\]

The cone $\sigma^k(\Nef(X))$ is generated by the 
classes corresponding to the columns of the matrix
\[
 \begin{pmatrix}
  1& k^2+k+1 & k^2-k+1 & 2k^2+1\\
  0& k+1&k&2k+1\\
  0&1&1&2
 \end{pmatrix},
\]
with respect to the basis $B$. 
We claim that the classes $\sigma^k(H)$ 
are extremal rays of the moving cone and generate it,
so that the following equality holds
\begin{wrapfigure}{r}{0.3\textwidth}
\includegraphics[width=0.42\textwidth]{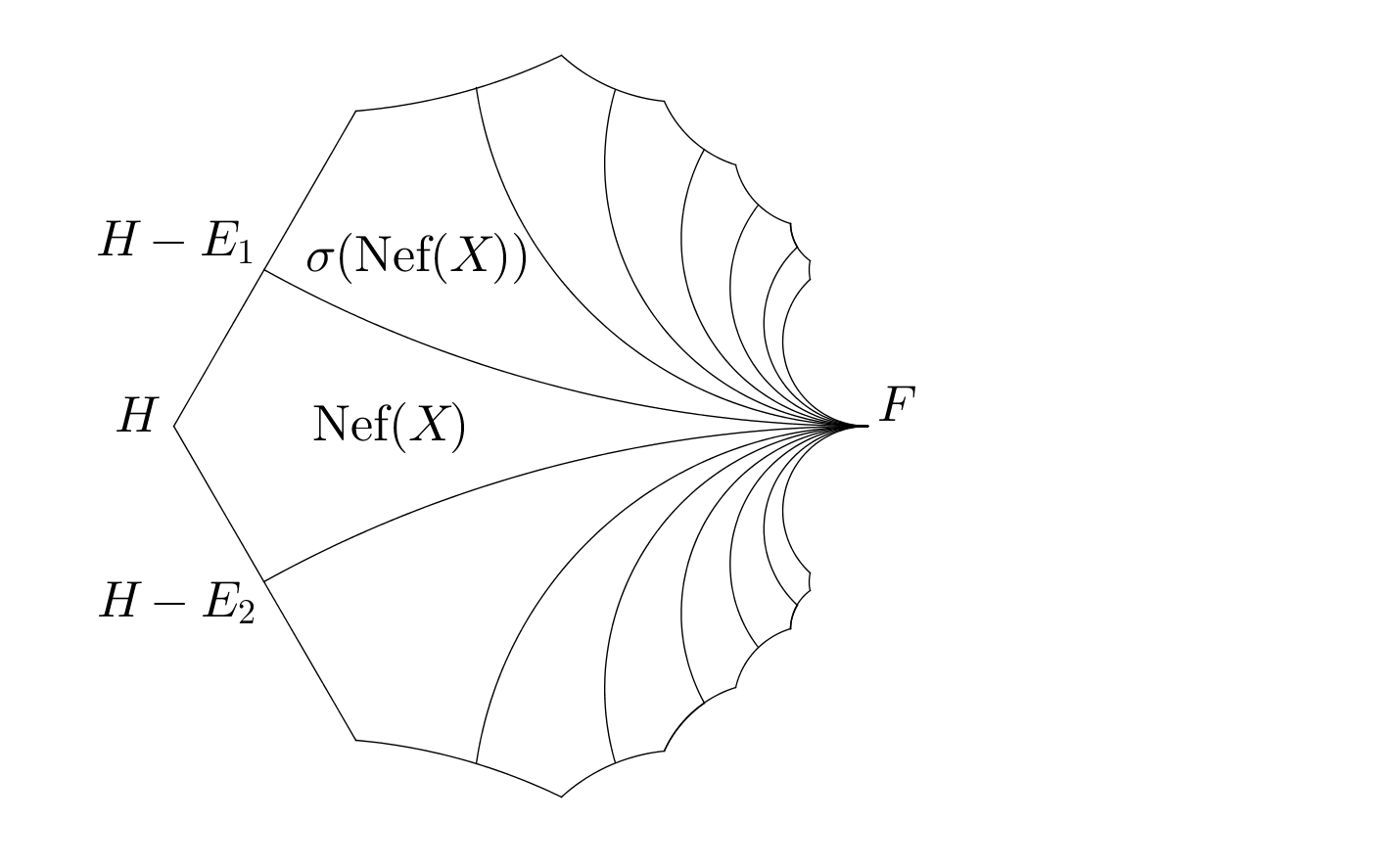}
\end{wrapfigure}
\[
 \Mov(X)
 =
 \bigcup_{k\in\zz}\sigma^k(\Nef(X)).
\]
First of all, observe that for each
$k\in\zz$ the cones $\sigma^k(\Nef(X))$ and 
$\sigma^{k+1}(\Nef(X))$ share the two-dimensional face 
generated by $F$ and $\sigma^k(H-E_1)=\sigma^{k+1}(H-E_2)$.
Moreover $\sigma^k(H)+\sigma^{k+1}(H)=
4\sigma^k(H-E_1)$, so that the union of the cones 
$\sigma^k(\Nef(X))$ is a convex cone and the classes 
$\sigma^k(H-E_i)$, $i=1,2$, are on its boundary 
but they are not extremal rays. Now observe that the right hand side cone is contained in $\Mov(X)$. 

Finally, since the property of lying on the boundary 
of $\Mov(X)$ is preserved by $\sigma^k$, 
we only have to prove that the two faces $\langle H,H-E_i\rangle$, 
for $i=1,2$, are on the boundary of the moving cone $\Mov(X)$.
We conclude observing that if we move outside from $\Nef(X)$ 
along a direction orthogonal to the face $\langle H,H-E_1\rangle$ 
(respectively $\langle H,H-E_2\rangle$) we obtain classes containing 
$E_2$ (respectively $E_1$) in the stable base locus.

\subsection{Generalized Bertini and Geiser involutions}
\label{flop}
We consider here a generalization of the classical
Bertini and Geiser involutions to blow-ups of del
Pezzo varieties. Let $Y$ be a degree $d\geq 3$ del Pezzo
variety and let $Z\subseteq Y$ be a zero-dimensional 
subscheme such that $\dim\langle Z\rangle=l(Z)-1$ and
the intersection of $d-1$ general elements of 
$\mathcal L_Z:=|\Osh_Y(1)\otimes\mathcal I_Z|$ 
is an elliptic curve. We denote by $\sigma\colon Y_Z\to Y$ 
the blow-up of $Y$ along $Z$ as in Section~\ref{sec:setup}.

If $Z$ has length $l(Z)=d-2$, then the general $(d-2)$-dimensional 
linear space containing $Z$ intersects $Y\setminus Z$ at two
distinct points. The birational involution obtained
by exchanging these two points induces
a birational involution $\sigma_G$ on
the blow-up $Y_Z$ of $Y$ at $Z$. We call 
this $\sigma_G$ a {\em generalized Geiser involution}.

When $Z$ has length $l(Z)=d-1$, denote by $F$ the 
divisor on $Y_Z$ defined as before. 
The base locus of the linear system $|F|$ consists of
one point $q$ while $|2F|$ defines a morphism 
$\varphi$. Since $F^n=1$, we have that $F^{n-1}$ 
is rationally equivalent to an elliptic curve $C$ passing 
through $q$ and the restriction $\varphi_{|C}$ is a double 
covering of a line passing through the point $p:=\varphi(q)$. 
Hence the image $\varphi(Y_Z)$ is a cone $V$. 
If we denote by $E$ the exceptional divisor corresponding
to the last blow-up of $\sigma$, we have that the restriction
$\varphi_{|E}$ is the $2$-veronese embedding $v_2$ of $\pp^{n-1}$.
We conclude that $V$ is a cone over $v_2(\pp^{n-1})$ and
$\varphi$ induces a birational involution $\sigma_B$ on $Y_Z$ that
we call a {\em generalized Bertini involution}.
We remark that if $X$ is  the del Pezzo elliptic variety 
obtained by blowing up $Y_Z$ in $q$, then 
$\sigma_B$ induces on $X$ the hyperelliptic involution
with respect to the origin given by the exceptional divisor.

\begin{remark}\label{geiser}
If $Y$ has degree four and the line $\langle Z\rangle$
does not contain any vertex $p_i$, then 
the indeterminacy locus of the corresponding 
Geiser involution $\sigma_G$ has codimension two. 
Moreover, it lifts to an isomorphism in codimension
one for the elliptic varieties of type 
$X_{21}$ and $X_{10}$. The action on the 
Picard group of $X$ in each case
is given by the following matrices respectively

{\small
\[
\sigma_{21}
=
\left(
\begin{array}{rrrrr}
3 & 1 & 1 & 0 & 0\\
-4 & -2 & -1 & 0 & 0\\
-4 & -1 & -2 & 0 & 0\\
0 & 0 & 0 & 0 & 1\\
0 & 0 & 0 & 1 & 0
\end{array}
\right)
\qquad
\sigma_{10}
=
\left(
\begin{array}{rrrrr}
3 & 1 & 1 & 0 & 0\\
-4 & -1 & -2 & 0 & 0\\
-4 & -2 & -1 & 0 & 0\\
0 & 0 & 0 & 1 & 0\\
0 & 0 & 0 & 0 & 1
\end{array}
\right).
\]
}
To prove this, we first claim that the lifted
birational map preserves the elliptic fibration
$\pi$ and thus it is a flop. Indeed, if $f$ is a fibre
of $\pi$ whose image $C$ in $Y$ is cut out by 
a three-dimensional linear space $L$ and we
fix a point $y\in C$, then 
the plane spanned by $y$ and $\langle Z\rangle$ 
is contained in $L$ and thus it intersects $C$ at a
fourth point, so that $\phi(f) = f$, which proves the claim.
Since $\phi$ preserves the fibration $\pi$, 
its pull-back $\phi^*$ must preserve both
the sets of horizontal and vertical divisors
of $X$. A direct calculation shows that the 
representative matrix for $\phi^*$ in the basis
$(H,E_1,\dots,E_4)$ is one of the above
in each case.
\end{remark}

\subsection{Mori chambers}
Let $X$ be a del Pezzo elliptic variety of degree four
with finite Mordell-Weil group. We provide here
the Mori chamber decomposition of the moving
cone $\Mov(X)$ of $X$. 
In the following proposition, we will denote by
$N$ the nef cone of $X$ and by
\[
 N_i
 :=
 {\rm cone}(\{F_I\, :\, i\in I\text{ and }F_I\in N\}
 \cup \{H - 2E_i\}).
\]

\begin{proposition}
\label{mc}
Let $X$ be a del Pezzo elliptic variety of degree four
such that the corresponding elliptic fibration has finite 
Mordell-Weil group. Then the Mori chamber decomposition 
of $\Mov(X)$ 
is given in the following table.
\begin{longtable}{|c|l|}
\hline
 Type & Cones\\
 \hline
$X_{43}$ & $N$, $N_1$, $N_2$, $N_3$, $N_4$\\ 
 \hline
$X_{22}$ & $N$, $N_1$, $N_2$\\ 
 \hline
$X_{21}$ & $N$, $N_1$, $N_2$, $\sigma_{21}^*(N)$, 
$\sigma_{21}^*(N_1)$, $\sigma_{21}^*(N_2)$\\ 
 \hline
$X_{11}$ & $N$, $N_1$\\ 
 \hline
$X_{10}$ & $N$, $N_1$, $\sigma_{10}^*(N)$, 
$\sigma_{10}^*(N_1)$\\
\hline
\end{longtable}
\end{proposition}
\begin{proof}
Let $X\to X_i$ be the flop of the class $h-e_i$ of the 
strict transform $C$ of a line through the point $q_i\in Y$.
Note that such a flop exists by~\cite{cps}.
We show that the nef cone of $X_i$ is $N_i$ and
then observe that the union of the cones in the table 
given in the statement is $\Mov(X)$ for each type.
To prove the claim, we begin by showing that the
primitive generators of the extremal rays of the
cone $N_i$ are nef in $X_i$. Observe that
each $F_I\in N_i$ is nef in both $X$ and $X_i$ since
$F_I\cdot (h-e_i) = 0$ by our definition of $N_i$. 
Hence we only have to check that also 
$H-2E_i$ is nef in $X_i$.
Since $H-2E_i$ is the pull-back of a class on 
the blow-up $\widetilde Y$ of $Y$ at $q_i$, it is enough to 
prove the claim on $\widetilde Y$.
By Lemma~\ref{cox:pt} the Cox ring of $\widetilde Y$
is finitely generated and the moving cone decomposes
as follows:
\[
 \Mov(\widetilde Y)
 =
 {\rm cone}(H,H-E_i) \cup {\rm cone}(H-E_i,H-2E_i).
\]
Thus, after flopping $h-e_i$ the class $H-2E_i$ becomes 
nef as claimed so that we have the inclusion $N_i\subseteq\Nef(X_i)$.
To prove that this is indeed an equality, we show that
the extremal rays of the dual cone of $N_i$ are classes of effective curves
of $X_i$. To this aim we make use of~\cite{cps}*{Lemma~4.1}
which asserts that if $\Gamma$ is a curve of $X$
which meets $C$ transversally at $k$ points, and no
other effective curve of class $h-e_i$, then the flop
image $\Gamma'$ of $\Gamma$ has class
\begin{equation}
\label{equ:flop}
 [\Gamma'] = [\Gamma] + k[C].
\end{equation}
By a direct calculation, we see that the extremal rays of the dual
cone of $N_i$ are the following (here we list only the
case $i=1$, being the remaining cases analogous):
{\tiny
\begin{longtable}{|c|l|l|}
\hline
& & \\
 Type & Extremal rays of the Mori cone 
 & Extremal rays of the dual cone of $N_1$ \\[2pt]
& &\\
\hline
&&\\
$X_{43}$ & 
$e_1,\,e_2,\,e_3,\,e_4$ & 
$-h+e_1,\,e_2,\,e_3,\,e_4,$\\
&
$h-e_1,\,h-e_2,\,h-e_3,\,h-e_4$ 
&
$2h-e_1-e_2,\,2h-e_1-e_3,\,2h-e_1-e_4$ \\
&&\\
\hline
&&\\
$X_{22},\, X_{21}$ &
$e_2,\,e_4$ &
$-h+e_1, e_2,\,e_4$\\
&
$h-e_1,\,h-e_3$ 
&
$2h-e_1-e_2,\ 2h-e_1-e_3 $\\
&
$e_1-e_2,\,e_3-e_4$
&
$e_3-e_4$
\\
&&\\
\hline
&&\\
$X_{11},\ X_{10}$ & $h-e_1,\ e_4,\ e_1-e_2$ & $-h+e_1,\ e_4, e_2-e_3,$ \\
& $e_2-e_3,\,e_3-e_4$ & $e_3-e_4,\ 2h-e_1-e_2$\\
& & \\
\hline
\end{longtable} }

For each type the curves having class $e_i$
or $e_i-e_{i+1}$, with $i>1$, do not intersect any curve of class 
$h-e_1$ and hence by~\eqref{equ:flop} their classes
in the Mori cone of $X_1$ are the same.
Assume that $\Gamma$ is
an irreducible curve such that
\[
 [\Gamma] = 2h-e_1-e_i.
\]
We can assume that $\Gamma$ is the
strict transform of a smooth conic $\mathcal C$
of $Y$
passing through $q_1$ and $q_i$, which
is possibly infinitely near to $q_1$.
The tangent line to $\mathcal C$ at $q_1$ 
cannot be contained in $Y$ since otherwise the 
plane spanned by $\mathcal C$ and this line
would be contained into each quadric of the 
pencil and thus in $Y$. This gives a contradiction, 
since the line through $q_1$ and $q_i$ is not 
contained in $Y$ by hypothesis.
Thus we conclude again by~\eqref{equ:flop},
proving the assertion for $X_{43}, X_{22}$ and $X_{11}$.

In case $X=X_{21}$ the chamber $\sigma_{21}^*(N)$
is the pull-back of the nef cone $N=\Nef(X)$ via 
the flop $\sigma_{21}$. Since $\sigma_{21}$ is
the generator of the Mordell-Weil group of $\pi$,
we deduce that $\sigma_{21}(X)$ is
an elliptic del Pezzo variety of the same type.
Thus each chamber $\sigma_{21}^*(N_i)$, for
$i=1,2$, is a flop image of $\sigma_{21}^*(N)$
exactly as $N_i$ is a flop image of $N$.
In particular the chamber $\sigma_{21}^*(N_i)$
is generated by finitely many semiample
classes of $\sigma_{21}(X_i)$.

Finally, in case $X=X_{10}$ we proceed as we did
for $X_{21}$, considering $\sigma_{10}$ instead
of $\sigma_{21}$.

\end{proof}

\begin{proof}[Proof of Theorem~$\ref{cox-mw}$]
By~\cite{hltu}*{Lemma~3.5} (1) implies (2), so let us suppose
that the Mordell-Weil group of $\pi$ is finite. If $d=1$ or $2$, then
we conclude by means of Theorem~\ref{cox:12}, while
the case $d=3$ has been proved in~\cite{hltu}*{Theorem~3.6}. 
Finally, when $d=4$, we observe that by Proposition~\ref{mc}, if 
the Mordell-Weil group of the fibration is finite, then 
the moving cone $\Mov(X)$ satisfies all the hypotheses of~\cite{HuKe}.

\end{proof}
\section{Cox rings}
\label{sec:cr}

In this section, we provide a presentation for
the Cox rings of the elliptic del Pezzo varieties
of degree $\leq 4$. We recall 
that given a normal projective variety $X$
with finitely generated picard group, its  
Cox ring $\R(X)$ can be defined as
(see~\cite{ADHL})
\[
\R(X)=
\bigoplus_{[D]\in \Pic(X)}
H^0(X,\Osh_X(D)).
\]
We apply~\cite{hkl}*{Algorithm 5.4} and we will explain all the steps
in the algorithm for the convenience of the reader.
Let $Y_1$ be a smooth projective variety with
finitely generated Cox ring $R_1$, which admits
a presentation $R_1 = \cc[T_1,\dots,T_{r_1}]/I_1$.
Note that $R_1$ is $K_1$-graded, where $K_1
=\Cl(Y_1)$.
Define $\overline Y_1 = \Spec(R_1)$ and let $\widehat Y_1
\subseteq\overline Y_1$ be the characteristic 
space of $Y_1$ with characteristic map 
$p\colon \widehat Y_1\to Y_1$. Let
$q\in Y_1$ be the point that we want to blow up.
We have the following commutative diagram:
\[
 \xymatrix{
  \overline{p^{-1}(q)}\ar[r] & \overline Y_1\\
  p^{-1}(q)\ar[r]\ar[d]^-p\ar[u] & \widehat Y_1\ar[d]^-p\ar[u]\\
  q\ar[r] & Y_1 .
}
\]
Let $I\subseteq R_1$ be the ideal of the closure
of $p^{-1}(q)$ in $\overline Y_1$, and let $J\subseteq R_1$ 
be the irrelevant ideal, i.e. the ideal of 
$\overline Y_1\setminus\widehat Y_1$.
We choose a system of homogeneous generators
$f_1,\dots,f_s\in R_1$ which form a basis for the ideal 
$I$ and such that
\begin{equation}
\label{step-1}
 f_i\in (I^{d_i}\,\colon\, J^\infty),
 \quad \forall i=1,\dots,s,
\end{equation}
where $d_i$ is a positive integer. An ample
class $[D]$ of $Y_1$ defines an embedding 
$Y_1\to Z_1$ into a projective toric variety
$Z_1$, whose Cox ring is the $K_1$-graded 
polynomial ring $\cc[T_1,\dots,T_{r_1}]$ and 
such that the class $[D]$ is ample on  $Z_1$. 
We embed $Z_1$ into another toric variety 
$W_1$ via the following map
\[
 (T_1,\dots,T_{r_1})
 \mapsto
 (T_1,\dots,T_{r_1},f_1,\dots,f_s),
\]
where the Cox ring of $W_1$ is the $K_1$-graded 
polynomial ring $\cc[T_1,\dots,T_{r_1+s}]$,
with $\deg(T_{r_1+i}) = \deg(f_i)$ for any $i$,
and again $[D]$ is an ample class of $W_1$.
Now we blow-up $W_1$ equivariantly along 
the orbit $V(T_{r_1+1},\dots,T_{r_1+s})$,
obtaining the toric variety $Z_2$ whose
Cox ring is the polynomial ring 
$\cc[T_1,\dots,T_{r_2}]$, where $r_2 =
r_1+s+1$, graded by the group
$K_2 := K_1 \oplus \zz$. 
Let $Y_2'\subseteq Z_2$  be the strict 
transform of the variety $Y_1$ as shown
in the following diagram
\[
 \xymatrix{
  Y_2'\ar[d]\ar[rr] && Z_2\ar[d]\\
  Y_1\ar[r] & Z_1\ar[r] & W_1.
}  
\]
Observe that $Y_2'$ is a blow-up 
(possibly weighted) of $Y_1$
at $q$, whose defining ideal 
is the following saturation
\begin{equation}
\label{step-2}
 I_2
 =
 \left(\langle T_{r_1+i}T_{r_2}^{d_i} - f_i
 \,\colon\,
 1\leq i\leq s\rangle
 +
 I_1\right)
 \,
 \colon
 \,
 \langle T_{r_2}\rangle
\end{equation}
with respect to the variable $T_{r_2}$.
Let $Y_2$ be the classical blow-up of
$Y_1$ at $q$.  To conclude that $Y_2' = Y_2$ 
and that $\cc[T_1,\dots,T_{r_2}]/I_2$ is isomorphic
to the Cox ring of $Y_2$, we need to check the
following inequality
\begin{equation}
\label{step-3}
 \dim I_2+\langle T_{r_2}\rangle
 >
 \dim I_2+\langle T_{r_2},T_\nu\rangle,
\end{equation}
where $T_\nu$ is the product of all the $T_i$'s,
for $1\leq i\leq r_1$, such that $V(T_i)$ does
not vanish identically at $p^{-1}(q)$.

\subsection{Degree one and two}
In this subsection, we provide a presentation
for the Cox rings of the del Pezzo elliptic varieties
of degree at most two with finite Mordell-Weil
group. Our main result is the following.

\begin{theorem}
\label{cox:12}
Let $\pi\colon X\to\pp^{n-1}$ be a del Pezzo elliptic 
fibration of degree $d\leq 2$ having finite Mordell-Weil
group.
Then the Cox ring of $X$ and its grading
matrix are listed in the following table:

{\footnotesize
\begin{longtable}{|c|l|c|}
\hline
& & \\
Type  & Cox ring  & Grading matrix\\
& &  \\
\hline
& & \\
$X_{1}$ 
&
$
\frac{\cc[T_1,\dots,T_{n+2},S]}
{\langle T_1^2-T_2^3+T_2\tilde f_4S^4 +\tilde f_6S^6\rangle}
$
&
$
\left[
 \begin{array}{rrrrrr}
  3&2&1&\dots&1&0\\
  0&0&-1&\dots&-1&1
 \end{array}
\right]
$
\\
& & \\
& $\tilde{f_d}:=f_d(T_1, T_2, T_3 S,\dots, T_{n+2} S)$ &\\
& & \\
\hline
& & \\
$X_{SS}$ 
&
$
 \frac{\cc[T_1,\dots,T_{n+3},S_1,S_2]}
 {\langle T_{n+2}S_1^4-T_{n+3}S_2^4+2\tilde{h}, 
 T_{n+2}T_{n+3}-\tilde{g}
 \rangle}
$
&
$
 \left[
 \begin{array}{rrrrrrrr}
  1&\dots&1&1 &2&2&0&0\\
  -1&\dots&-1& 0 &-4&0&1&0\\
  -1&\dots&-1& 0& 0&-4&0&1\\
 \end{array}
 \right]
$
\\
& & \\
& $\tilde h := h(T_1S_1S_2,\dots,T_nS_1S_2,T_{n+1})$, &\\
& $\tilde{g} := g(T_1,\dots,T_n)$ &\\
& & \\
\hline
& & \\
$X_{2}$ 
&
$
 \frac{\cc[T_1,\dots,T_{n+2},S_1,S_2]}
 {\langle T_{n+2}^2-S_2^2\tilde{f} -T_nT_{n+1}^3\rangle}
$
&
$
\left[
 \begin{array}{rrrrrrrr}
   1&\dots& 1& 1&1& 2&0&0\\
  -1&\dots&-1&-2&0&-1&1&0\\
  -1&\dots&-1&0&0&0&-1&1
 \end{array}
 \right]
$
\\
& & \\
& $\tilde{f} := \frac
{f(T_1S_1S_2^2,\dots,T_{n-1}S_1S_2^2,T_nS_1^2S_2^2,T_{n+1})}
{S_1^2S_2^2}$ & \\
& & \\
\hline
\end{longtable}
}
\end{theorem}
\begin{proof}
In order to prove the case $X_1$, let $Y_1$ be 
the del Pezzo variety given
by the polynomial~\eqref{type:1} and
let $q\in Y_1$ be the point of coordinates
$(1,1,0,\dots,0)$. The ring $R_1$ equals
$\cc[T_1,\dots,T_{n+2}]/I_1$, where
$I_1$ is the principal ideal generated 
by the polynomial~\eqref{type:1}.
We take $I, J\subseteq R_1$ as before and 
we choose the following homogenous 
elements $f_1,\dots,f_{n}$:
\[
 T_3,\dots,T_{n+2}\in I
\]
as in~\eqref{step-1}, i.e. all of them 
have $d_i = 1$.
Observe that the saturated ideal~\eqref{step-2} is
\[
 I_2
 =
 \langle T_{n+2+i}T_{2n+5} - f_i
 \,\colon\, 1\leq i\leq n\rangle
 +
 I_1
\]
since, after applying the substitution 
$T_{2+i} = T_{n+2+i}T_{2n+5}$ for each $i=1,\dots, n+2$,
the resulting polynomial $T_1^2-T_2^3+
T_2\tilde f_4T_{2n+5}^4+ \tilde f_6T_{2n+5}^6$
is not divisible by $T_{2n+5}$.
Finally, according to~\eqref{step-3}, we need to check that
\[
 \dim I_2+\langle T_{2n+5}\rangle
 >
 \dim I_2+\langle T_{2n+5},T_1T_2\rangle,
\]
and this is easily checked to hold, being $I_1$ 
a principal ideal.
We conclude that the ring $\cc[T_1,\dots,T_{2n+5}]/I_2$
is isomorphic to the the Cox ring of the 
blow-up $X_1$ of $Y$ at $q$. After eliminating
the fake linear relations and renaming the variables,
we get the claimed presentation for the Cox ring.

We now prove the case $X_{SS}$.
Let $Y_1$ be the del Pezzo variety given
by the polynomial~\eqref{type:ss} and
let $q\in Y_1$ be the point of coordinates
$(0,\dots,0,1,1)$. The ring $R_1$ equals
$\cc[T_1,\dots,T_{n+2}]/I_1$, where
$I_1$ is the principal ideal generated 
by the polynomial~\eqref{type:ss}.
We take $I,J\subseteq R_1$ 
as before and choose the following homogenous 
elements $f_1,\dots,f_{n+1}$:
\[
 T_1,\dots,T_{n}\in I,
 \qquad
 T_{n+2}-h\in (I^4\colon J^\infty)
\]
as in~\eqref{step-1}, that is the first $n$ sections
have $d_i = 1$, while $d_{n+1} = 4$.
Observe that the ideal in~\eqref{step-2} is
\[
 I_2
 =
 \langle
 T_{n+2+i}T_{2n+4}^{d_i} - f_{i}
 \,\colon\, 1\leq i\leq n+1
 \rangle
 + 
 \langle
 T_{2n+3}^2T_{2n+4}^4+2T_{2n+3} h' - g'
 \rangle,
\]
where $h' = h(T_{n+3}T_{2n+4},\dots,
T_{2n+2}T_{2n+4},T_{n+1})$ and  $g' = 
g(T_{n+3},\dots,T_{2n+2})$.
According to~\eqref{step-3}, we can easily check that the 
following inequality holds:
\[
 \dim I_2+\langle T_{2n+4}\rangle
 >
 \dim I_2+\langle T_{2n+4},T_{n+1}T_{n+2}\rangle.
\]
Thus, after eliminating the fake 
linear relations from $I_2$ and renaming the variables,
we can conclude that the Cox ring and the grading matrix of the blow-up 
$Y_2$ of $Y_1$ at $q$ are the following
\[
 R_2 
 =
 \frac{\cc[T_1,\dots,T_{n+2},S]}
 {\langle T_{n+2}^2S^4+2 h''T_{n+2}- g''\rangle}
 \qquad
 \left[
  \begin{array}{rcrrrr}
   1&\dots&1&1&2&0\\
   -1&\dots&-1&0&-4&1
  \end{array}
 \right]
\]
where $h'' =  h(T_1S,\dots,
T_n S,T_{n+1})$ and  $g'' = 
g(T_1,\dots,T_n)$.
The irrelevant ideal is $J_2 = \langle T_1,\dots,T_n,T_{n+2}\rangle
\cap\langle T_{n+1},S\rangle$.
We now repeat the 
procedure blowing-up $Y_2$ at the point $q_2'$ which lies
over $q_2 = (0,\dots,0,1,-1)\in Y$. Recall that there is a
$\cc^*$-equivariant embedding of total coordinate spaces
\[
 \overline Y_1\to\overline{Y_2}
 \qquad
 (T_1,\dots, T_{n+2})\to (T_1,\dots,T_{n+1},T_{n+2} - h,1)
\]
which induces the birational map $Y_1\dashrightarrow Y_2$.
The image of $q_2$ is the point of homogeneous
coordinates $q_2' = (0,\dots,0,1,-2,1)$.
We choose the following homogenous 
elements $f_1,\dots,f_{n+2}$:
\[
 T_1,\dots,T_{n},2T_{n+1}^2+T_{n+2}S^4\in I,
 \qquad
 T_{n+2}S^4+2h'' \in (I^4\colon J^\infty)
 \]
as in~\eqref{step-1}, that is the first $n+1$ sections
have $d_i = 1$, while $d_{n+2} = 4$.
The ideal in~\eqref{step-2} is
\[
 I_3
 =
 \langle
 T_{n+3+i}T_{2n+6}^{d_i} - f_{i}
 \,\colon\, 1\leq i\leq n+2
 \rangle
 +
 \langle
 T_{n+2}^2S^4+2T_{n+2} h''' - T_{2n+6}^4 g'''
 \rangle
\]
where $h''' =  h(T_{n+4}T_{2n+6}S,\dots,
T_{2n+3}T_{2n+6} S,T_{n+1})$ and  $g''' = 
g(T_{n+4},\dots,T_{2n+3})$.
After eliminating the fake linear relations from
the above ideal and renaming 
the variables, we get 
the statement for the case $X_{SS}$.\\

Finally, let us prove the case $X_{2}$.
Let $Y_1$ be the del Pezzo variety given
by the polynomial~\eqref{type:2} and
let $q\in Y_1$ be the point of coordinates
$(0,\dots,0,1,0)$. The ring $R_1$ equals
$\cc[T_1,\dots,T_{n+2}]/I_1$, where
$I_1$ is the principal ideal generated 
by the polynomial~\eqref{type:2}.
We take $I,J\subseteq R_1$
as before and choose the following homogenous 
elements $f_1,\dots,f_{n+1}$:
\[
 T_1,\dots,T_{n-1},T_{n+2}\in I,
 \qquad
 T_{n}\in (I^2\colon J^\infty)
\]
as in~\eqref{step-1}, that is the first $n$ sections
have $d_i = 1$, while $d_{n+1} = 2$.
Observe that the ideal in~\eqref{step-2} is
\[
 I_2
 =
 \langle
 T_{n+2+i}T_{2n+4}^{d_i} - f_{i}
 \,\colon\, 1\leq i\leq n+1
 \rangle
 + 
 \langle
 T_{2n+2}^2-f' -T_{2n+3}T_{n+1}^3
 \rangle
\]
with $f' = T_{2n+4}^{-2}f(T_{n+3}T_{2n+4},\dots,T_{2n+1}T_{2n+4},
T_{n}T_{2n+4}^2,T_{n+1})$.
According to~\eqref{step-3} it can be easily checked that 
\[
 \dim I_2+\langle T_{2n+4}\rangle
 >
 \dim I_2+\langle T_{2n+4},T_{n+1}\rangle.
\]
Thus, after eliminating the fake 
linear relations from $I_2$ and renaming the variables,
we conclude that the Cox ring and the grading matrix of the blow-up 
$Y_2$ of $Y_1$ at $q$ are the following
\[
 R_2 
 =
 \frac{\cc[T_1,\dots,T_{n+3}]}
 {\langle T_{n+2}^2-f'' -T_{n}T_{n+1}^3\rangle}
 \qquad
 \left[
  \begin{array}{rcrrrrr}
   1&\dots&1&1&1&2&0\\
   -1&\dots&-1&-2&0&-1&1
  \end{array}
 \right]
\]
with $f'' = T_{n+3}^{-2}f(T_1T_{n+3},\dots,T_{n-1}T_{n+3},
T_{n}T_{n+3}^2,T_{n+1})$.
The irrelevant ideal of $R_2$ is $J_2 = \langle T_1,\dots,T_{n-1},T_{n+2}\rangle
\cap\langle T_{n},S\rangle$.
Now repeat the 
procedure by blowing up $Y_2$ at the point 
$q_2' = (0,\dots,0,1,1,1,0)$ which is 
the invariant point with respect
to the lifted involution 
$(T_1,\dots,T_{n+3})\mapsto (T_1,\dots,T_{n+1},-T_{n+2},T_{n+3})$,
and it corresponds to the generator
of the kernel of the differential 
$d\varphi_q$.
We choose the following homogenous 
elements $f_1,\dots,f_{n}$:
\[
 T_1,\dots,T_{n-1},T_{n+3}\in I,
 \]
as in~\eqref{step-1}, i.e. $d_i = 1$
for all the sections.
The ideal in~\eqref{step-2} is
\[
 I_3
 =
 \langle
 T_{n+3+i}T_{2n+3}^{d_i} - f_{i}
 \,\colon\, 1\leq i\leq n
 \rangle
 +
 \langle
 T_{n+2}^2-T_{2n+3}^2\tilde f -T_{n}T_{n+1}^3
 \rangle
\]
where $\tilde f =  T_{2n+3}^{-2}
f''(T_{n+4}T_{2n+3}S,\dots,T_{2n+3}T_{2n+3} S,T_{n+1})$.
After eliminating the fake linear relations from
the above ideal and renaming 
the variables, we obtain 
the statement for the case $X_{2}$.

\end{proof}

\subsection{Degree four}
In this last subsection, we first provide the 
following presentation
for the Cox rings of the blowing-up of a del
Pezzo variety of degree four at a point.
\begin{lemma}
\label{cox:pt}
Let $Y$ be a smooth complete intersection 
of two quadrics of $\pp^{n+2}$. After possibly
applying a linear change of coordinates, the ideal
of $Y$ is generated by 
$x_2x_3-x_1x_2+f(x_4,\dots,x_{n+3})$ and
$x_2x_3-x_1x_3+g(x_4,\dots,x_{n+3})$.
The blow-up $\widetilde Y$ of $Y$ at the point 
$q=(1,0,\dots,0)\in Y$
has the following Cox ring and grading matrix
{\small
\[
\frac{\cc[T_1,\dots,T_{n+3},S]}
 {\langle T_2T_3S^2-T_1T_2+f,
T_2T_3S^2-T_1T_3+g
\rangle}
\quad\quad
\left[
 \begin{array}{rrrrrrr}
   1 & 1  & 1 & 1& \dots& 1 & 0\\
   0 & -2 &-2 &-1& \dots &-1 & 1
 \end{array}
 \right]
\]
}
respectively, where $f=f(T_4,\dots,T_{n+3})$ and
$g=g(T_4,\dots,T_{n+3})$.
\end{lemma}
\begin{proof}
After applying a linear change of coordinates, we
can assume that $q$ is a point of $Y= Q\cap Q'$, where $Q$ is singular at 
$(1,1,0,0,\dots,0)$ and $Q'$ is singular at 
$(1,0,1,0,\dots,0)$, and that the tangent hyperplanes
to $Q$ and $Q'$ at $q$ are $V(x_2)$ and $V(x_3)$,
respectively. This proves the first claim.

To prove the second statement, we take 
$R_1$ to be $\cc[T_1,\dots,T_{n+3}]/I_1$, where 
$I_1$ is the ideal of $Y$, and we apply~\cite{hkl}*{Algorithm 5.4}.
We take $I,J\subseteq R_1$
as before and choose the following homogenous 
elements $f_1,\dots,f_{n+2}$:
\[
 T_4,\dots,T_{n+3}\in I,
 \qquad
 T_{2}, T_3\in (I^2\colon J^\infty)
\]
as in~\eqref{step-1}, that is the first $n$ sections
have $d_i = 1$, while $d_{n+1} = d_{n+2} = 2$.
The ideal in~\eqref{step-2} is
\[
\begin{array}{rcl}
 I_2 &
 = 
 &
 \langle
 T_{n+3+i}T_{2n+6}^{d_i} - f_{i}
 \,\colon\, 1\leq i\leq n+2
 \rangle +\\
&& + 
 \langle T_{2n+4}T_{2n+5}T_{2n+6}^2-T_1T_{2n+4}+\tilde{f},
T_{2n+4}T_{2n+5}T_{2n+6}^2-T_1T_{2n+5}+\tilde{g}\rangle,
\end{array}
\]
where $\tilde{f} := f(T_{n+4},\dots,T_{2n+3})$ and
$\tilde{g} := g(T_{n+4},\dots,T_{2n+3})$.
According to~\eqref{step-3} it can be easily checked that 
\[
 \dim I_2+\langle T_{2n+6}\rangle
 >
 \dim I_2+\langle T_{2n+6},T_{1}\rangle.
\]
After eliminating the fake linear relations from $I_2$ and 
renaming the variables, we get the second statement.

\end{proof}

We conclude with two examples of the computation of the
Cox rings for the del Pezzo elliptic varieties of degree four and
dimension three. 
We only report the final results, since the computations have
been done with the same procedure as before.

\vspace{5mm}

{\tiny
\begin{tabular}{l}
{\bf Case $X_{43}$}:\\
\\
\hline
\\
Equations:\\
\\
$
\begin{array}{l}
    x_{2}^2 - x_{3}^2 + x_{4}^2 + x_{5}^2 + x_{6}^2,\\
    x_{1}^2 - x_{3}^2 + 2x_{4}^2 + 3x_{5}^2 + 4x_{6}^2
\end{array}
$
\\
\\
\hline
\\
Cox ring $\cc[T_1,\dots,T_{13}]/I$, where $I$ is generated by:\\
\\
$
\begin{array}{l}
    T_{1}^2 - T_{3}^2 + 2T_{5}T_{8} - T_{6}T_{7},\\
    T_{2}^2 + 2T_{3}^2 - T_{5}T_{8} + T_{6}T_{7},\\
    T_{4}T_{9} - T_{5}T_{8} - T_{6}T_{7},\\
    T_{4}T_{10}^2 - T_{7}T_{12}^2 + T_{8}T_{13}^2,\\
    T_{4}T_{11}^2 - T_{5}T_{12}^2 - T_{6}T_{13}^2,\\
    T_{5}T_{8}T_{11}^2T_{12}^2T_{13}^2 - T_{5}T_{9}T_{12}^4T_{13}^2 + 
        T_{6}T_{7}T_{11}^2T_{12}^2T_{13}^2 - T_{6}T_{9}T_{12}^2T_{13}^4,\\
    T_{5}T_{10}^2 - T_{7}T_{11}^2 + T_{9}T_{13}^2,\\
    T_{6}T_{10}^2 + T_{8}T_{11}^2 - T_{9}T_{12}^2
\end{array}
$
\\
\\
\hline
\\
Degree matrix:\\
\\
$
\left[
\begin{array}{rrrrrrrrrrrrr}
 1 & 1 &  1 & 1 &  1 & 1 &  1 & 1 &  1 & 0 &  0 & 0 &  0\\
-1 & -1 & -1 & -2 & -2 & -2 &  0 & 0 &  0 & 1 &  0 & 0 &  0\\
-1 & -1 & -1 & -2 &  0 & 0 & -2 & -2 &  0 & 0 &  1 & 0 &  0\\
-1 & -1 & -1 & 0 & -2 & 0 & -2 & 0 & -2 & 0 &  0 & 1 &  0\\
-1 & -1 & -1 & 0 &  0 & -2 &  0 & -2 & -2 & 0 &  0 & 0 &  1\\
\end{array}
\right]
$
\\
\\
\hline
\\
\\
\end{tabular}
}

{\tiny 
\begin{tabular}{l}
{\bf Case $X_{22}$}:\\
\\
\hline
\\
Equations:\\
\\
$
\begin{array}{l}
     x_{2}^2 - x_{3}^2 + x_{4}^2 + x_{5}^2 + x_{6}^2,\\
    x_{1}^2 - x_{3}^2 + \frac34x_{4}^2 + 2x_{5}^2 + 3x_{6}^2
\end{array}
$
\\
\\
\hline
\\
Cox ring $\cc[T_1,\dots,T_{10}]/I$, where $I$ is generated by:\\
\\
$
\begin{array}{l}
    2T_{1}^2 - T_{3}^2T_{8}^2T_{10}^2 + 11T_{4}^2T_{7}^2T_{8}^2T_{9}^2T_{10}^2 - 
        8T_{4}T_{5}T_{7}^2T_{8}^4 + 8T_{4}T_{6}T_{9}^2T_{10}^4 - 4T_{5}T_{6}T_{8}^2T_{10}^2,\\
    2T_{2}^2 + 3T_{3}^2 + 3T_{4}^2T_{7}^2T_{9}^2 + 4T_{5}T_{6}
\end{array}
$
\\
\\
\hline
\\
Degree matrix:\\
\\
$
\left[
\begin{array}{rrrrrrrrrr}
 1 & 1 &  1 & 1 &  1 & 1 &  0 & 0 &  0 & 0 \\
-1 & -1 & -1 & -2 & -2 & 0 &  1 & 0 &  0 & 0 \\
-1 & -1 & -1 & -2 &  0 & -2 &  0 & 0 &  1 & 0 \\
 0 & -1 & -1 & 0 & -2 & 0 & -1 & 1 &  0 & 0 \\
 0 & -1 & -1 & 0 &  0 & -2 &  0 & 0 & -1 & 1 \\
\end{array}
\right]
$
\\
\\
\hline
\\
\\
\end{tabular}
}

\end{document}